\documentclass[11pt]{amsart}

\usepackage[utf8]{inputenc}
\usepackage[T1]{fontenc}
\usepackage[english]{babel}

\usepackage{hyperref}
\usepackage{amsthm}
\usepackage{amsmath}
\usepackage{amssymb}
\usepackage{mathrsfs}
\usepackage{bbold}
\usepackage{graphicx}
\usepackage{listings}
\usepackage{mathrsfs}
\usepackage{enumerate}
\usepackage{pict2e}
\usepackage{stmaryrd}
\usepackage{mathtools}
\usepackage{mathabx}

\usepackage[affil-it]{authblk}

\usepackage[all,cmtip]{xy}

\newtheorem{theorem}{Theorem}[section]
\newtheorem{corollary}{Corollary}

\newtheorem{lemma}[theorem]{Lemma}
\newtheorem{proposition}{Proposition}

\theoremstyle{definition}

\DeclareMathOperator{\ug}{\boldsymbol{u}}

\DeclareMathOperator{\HDNLS}{H_{ \rm DNLS}}

\DeclareMathOperator{\sinc}{sinc}
\DeclareMathOperator{\domega}{d\omega}
\DeclareMathOperator{\dx}{dx}
\DeclareMathOperator{\ds}{ds}

\DeclareMathOperator{\supp}{supp}

\newcommand\unupoop[3]{\mathrel{\mathop{#1}\limits_{#2}^{#3} }  }

\author{Joackim \textsc{Bernier}}

\email{joackim.bernier@univ-rennes1.fr}


\title[Bounds on the high Sobolev norms of DNLS]{Bounds on the growth of high discrete Sobolev norms for the cubic discrete nonlinear Schr\"odinger equations on $h\mathbb{Z}$.}

\begin{document}

\maketitle

\centerline{\scshape Joackim Bernier}
\medskip
{\footnotesize
 \centerline{IRMAR, CNRS UMR 6625}
 \centerline{Universit\'e de Rennes 1, Campus de Beaulieu}
   \centerline{ 263 avenue du G\'en\'eral Leclerc, CS 74205 }
   \centerline{ 35042 Rennes cedex, France}
}

\begin{abstract}
We consider the discrete nonlinear Schr\"odinger equations on a one dimensional lattice of mesh $h$, with a cubic focusing or defocusing nonlinearity.
We prove a polynomial bound on the growth of the discrete Sobolev norms, uniformly with respect to the stepsize of the grid. This bound is based on a construction of higher modified energies.
\end{abstract}

\section{Introduction}
We consider the cubic discrete nonlinear Schr\"odinger equation (called DNLS)  on a grid $h\mathbb{Z}$ of stepsize $h>0$. This equation is a differential equation on $\mathbb{C}^{h\mathbb{Z}}$ defined by (see \cite{MR2742565} and the references therein for details about its derivation)
\begin{equation}
\label{def_DNLS}
\forall g\in h\mathbb{Z}, \ i\partial_t \ug_g = (\Delta_h \ug)_g+ \nu |\ug_g|^2 \ug_g,
\end{equation}
where $\nu \in \{-1,1\}$ is a parameter and $\Delta_h \ug$ is the discrete second derivative of $\ug$. It is defined by
\[\forall g\in h\mathbb{Z}, \ (\Delta_h \ug)_g = \frac{\ug_{g+h}  - 2 \ug_g + \ug_{g-h}}{h^2}. \]
We consider both the \it focusing \rm and the \it defocusing \rm equations. They correspond respectively to the choices $\nu = 1$ and $\nu =-1$.

\medskip

DNLS is a popular model in numerical analysis for the spatial discretization of the cubic nonlinear Schr\"odinger equation (NLS), given by:
\begin{equation}
\label{def_NLS}
i\partial_t u = \partial_x^2 u + \nu |u|^2u, 
\end{equation}
see, for example, \cite{MR3018143},\cite{MR2576378},\cite{ref_traveling_wave},\cite{MR2277106},\cite{MR3460749},\cite{MR2742565}. 
Motivated by the approximation properties of NLS by DNLS, we consider the discrete model near its continuous limit {\em i.e.} when $h$ goes to $0$.
So, we introduce norms consistent with the usual continuous norms and we pay attention to establish estimates uniform with respect to $h$.

\medskip

We introduce the discrete $L^2$ space. It is defined by
\[ L^2(h\mathbb{Z}) = \left\{  \ug \in \mathbb{C}^{h\mathbb{Z}}, \  \|\ug\|_{L^2(h\mathbb{Z})}^2 = h \sum_{g\in h\mathbb{Z}} |\ug_g|^2< \infty  \right\} .\] 
This space is natural to solve DNLS. Indeed, as $L^2(h\mathbb{Z})$ is a Banach algebra (which is not the case in the continuous setting), Cauchy Lipschitz Theorem can be applied to get the local well posedness of DNLS in $L^2(h\mathbb{Z})$. Furthermore, since \eqref{def_DNLS} is invariant by gauge transform, as a consequence of the Noether Theorem the discrete $L^2$ norm is a constant of the motion of DNLS. Thus, DNLS is globally well posed in $L^2(h\mathbb{Z})$. 

\medskip

We introduce the homogeneous discrete Sobolev norms by analogy with respect to the continuous homogeneous Sobolev norms. If $n\in \mathbb{N}$ is an integer and $\ug \in L^2(h\mathbb{Z})$, its discrete homogeneous Sobolev norm of order $n$ is defined by
\begin{equation}
\label{def_high_Sobolev}
\| \ug \|_{\dot H^n(h\mathbb{Z})}^2 = \langle (-\Delta_h)^n \ug , \ug \rangle_{L^2(h\mathbb{Z})}.
\end{equation}
For example, if $\ug \in L^2(h\mathbb{Z})$, its discrete homogeneous Sobolev norm of order $1$ is
\[ \| \ug \|_{\dot H^1(h\mathbb{Z})} =\sqrt{ h \sum_{g\in h\mathbb{Z}} \left| \frac{\ug_{g+h}-\ug_g}h \right|^2}. \]
Naturally, we define as usual the non homogeneous discrete Sobolev norms by
\[ \| \ug \|_{H^n(h\mathbb{Z})}^2 = \sum_{k=0}^n \| \ug \|_{\dot H^k(h\mathbb{Z})}^2.\]
\medskip

Applying the triangle inequality we can easily prove that all these norms are controlled by the discrete $L^2$ norm
\begin{equation}
\label{trivial_est}
 \forall \ug \in L^2(h\mathbb{Z}), \ \| \ug \|_{\dot H^n(h\mathbb{Z})} \leq \left( \frac2h \right)^n \|\ug \|_{L^2(h\mathbb{Z})}.
\end{equation}
So, since the discrete $L^2$ norm is a constant of the motion of DNLS, any discrete Sobolev norm of a solution of DNLS is globally bounded. However, this bound is not uniform with respect to the stepsize $h$. Consequently, these estimates are trivial when we consider the continuous limit. 

\medskip

An uniform control of these norms with respect to $h$ may be crucial to establish aliasing\footnote{Aliasing usually refers to a default of commutation between a nonlinearity and an interpolation.} or consistency estimates. For example, in \cite{ref_traveling_wave},  the existence and the stability of traveling waves is studied near the continuous limit of the focusing DNLS. The discrete Sobolev norms are used to control an aliasing error generated by the variations of the momentum (see Theorem $1.5$ of \cite{ref_traveling_wave}). It is proven that if for all $n\in \mathbb{N}$, the discrete Sobolev norm of order $n$ of the solutions of the focusing DNLS can be bounded by $t^{\alpha_n}$, uniformly with respect to $h$, then DNLS admits solutions whose behavior is similar to traveling waves for times of order $h^{-\beta}$, with $\beta =\limsup_n \frac{n}{\alpha_n}$.

\medskip

There is a huge literature about the growth of the Sobolev norms for continuous Schr\"odinger equations. Since, we are focusing on the continuous limit of DNLS, it is natural to try to adapt the methods used for these equations. If we focus on the continuous Schr\"odinger equations on $\mathbb{R}$, it seems that there are three families of methods and results.

\begin{itemize}
\item First, there is the cubic nonlinear Schr\"odinger equation. This equation is known to be \it completely integrable\rm. In particular, it admits a sequence of constants of the motion coercive in $H^n(\mathbb{R})$. Consequently, all the Sobolev norms are globally bounded (see, for example, \cite{MR2996998}).
\item Second, there is the linear Schr\"odinger equation with a potential smooth with respect to $t$ and $x$. In such case, for all $\varepsilon>0$ there is a control of the growth by $t^{\varepsilon}$ (see \cite{MR1753490}).
\item Third, in the other cases, there are methods using dispersion and/or higher modified energy. They were first introduced by Bourgain \cite{MR1386079} and continued in the work of Staffilani \cite{MR1427847}. They provide a control of the growth of the  $H^n$ norm by $t^{\alpha n + \beta}$ for some $\alpha,\beta \in \mathbb{R}$. More recently, applying these methods \cite{MR2996998},  Sohinger  proves a control of the $H^s$ norm by $t^{\frac13 s+}$ for the nonlinear Schr\"odinger equation with an Hartree nonlinearity.
\end{itemize} 

\medskip

A priori, DNLS is not a completely integrable equation, so we can not control its Sobolev norms as for its continuous limit (for a completely integrable spatial discretization of NLS, we can refer to the Ablowitz-Ladik model, see \cite{MR2040621}). In this paper, we adapt the last method to the discrete nonlinear Schr\"odinger equation. In \cite{MR2150357}, Stefanov and Kevrekidis proved that the dispersion is weaker for the linear discrete Schr\"odinger equation than for the continuous equation. They got a $L^\infty$ decay of the form $t^{-\frac12} + (ht)^{-\frac13}$ (see also \cite{MR2277106}). Using dispersive arguments in our setting seems thus more difficult than in the continuous case and does not seem to strengthen significantly the results. However,  the method of constructing \it modified energies \rm can be applied and turns out to yield results comparable to the continuous case  (i.e. a polynomial bound whose exponent is proportional to the index of the Sobolev norm detailed above). More precisely, with  our construction, we get the following bound.

\begin{theorem}
\label{thm_growth}
  For all $n\in \mathbb{N}^*$, there exists $C>0$, such that for all $h>0$ and all $\nu\in \{-1,1\}$,  if $\ug \in C^1(\mathbb{R};L^2(h\mathbb{Z}))$ is a solution of DNLS  then for all $t\in \mathbb{R}$
  \begin{equation}
  \label{ineq_growth}
 \| \ug(t) \|_{\dot{H}^{n}(h\mathbb{Z})} \leq C \left[ \| \ug(0) \|_{\dot{H}^{n}(h\mathbb{Z})} +   M_{\ug(0)}^{\frac{2n+1}3}   +  |t|^{\frac{n-1}2} M_{\ug(0)}^{\frac{4n-1}3}  \right] ,  
\end{equation}
where
\[  M_{\ug(0)} = \|\ug(0) \|_{\dot{H}^1(h\mathbb{Z})} +\|\ug(0) \|_{L^2(h\mathbb{Z})}^3  .\]
\end{theorem}

This theorem is the main result of this paper, it will be proven in the third section. The second section is devoted to the introduction of tools and notations useful to prove it.

\medskip

We conclude this introduction with some remarks about estimate \eqref{ineq_growth}. 
\begin{itemize}
\item If $n=1$ then the discrete $H^1$ norm is globally bounded, uniformly with respect to $h$. It is a consequence of the conservation of the Hamiltonian of DNLS and its coercivity in $H^1(h\mathbb{Z})$. In the focusing case this argument is specific to the dimension $1$. It is based on a discrete  Gagliardo-Nirenberg inequality. For the defocusing case, the coercivity is straighforward and can be extended to higher dimensions and with other nonlinearities.
\item The factor associated to the growing term $t^{\frac{n-1}2}$ is $M_{\ug(0)}^{\frac{4n-1}3}$. So the growth of the high Sobolev norms is controlled by the size of the initial condition with respect to the low Sobolev norms.
\item The estimate \eqref{ineq_growth} is homogeneous. More precisely, DNLS is invariant by dilatation in the sense that if $\ug$ is a solution of DNLS then $(t,g) \mapsto \lambda \ug_{\lambda g}(\lambda^2 t )$ is a solution of DNLS with stepsize $h\lambda^{-1}$. Estimate \eqref{ineq_growth} is invariant by this transformation (as can be seen from the exponents of $M_{\ug(0)}$). Consequently, to prove Theorem \ref{thm_growth}, we just have to prove it with $h=1$.
\item   Here, the construction of higher modified energies relies essentially algebraic considerations. In particular, it does not use any dispersion effect. So it seems possible to realize almost the same proof to get estimate \eqref{ineq_growth} with periodic boundary conditions.

\end{itemize}

\section{Shannon interpolation}
In order to use classical analysis tools, it is very useful to identify sequences of $L^2(\mathbb{Z})$ with functions defined on the real line through an interpolation method. Here, we choose the Shannon interpolation (this choice is quite natural, see \cite{MR2277106} or \cite{ref_traveling_wave}). More precisely, it is the usual interpolation  we get extending a sequence into a real function whose Fourier transform is supported on $[-\pi,\pi]$.

\medskip

In this section, we introduce this interpolation and we give some of its classical properties useful to prove Theorem \ref{thm_growth}. For details or proofs of these classical properties the reader can refer to \cite{ref_traveling_wave} or \cite{MR3156669}.

\medskip

First we need to define the \it discrete Fourier transform \rm
\begin{equation}
\label{def_discrete_FT}
\mathcal{F} : \left\{  \begin{array}{llll} L^2(\mathbb{Z}) & \to & L^2(\mathbb{R} / 2\pi \mathbb{Z}) \\
																\ug & \mapsto &\displaystyle \omega \mapsto \sum_{g \in \mathbb{Z}} \ug_g e^{ i g \omega}
\end{array} \right. ,
\end{equation}
 and the \it Fourier Plancherel transform \rm
\[ \mathscr{F} : \left\{  \begin{array}{llll} L^2(\mathbb{R}) & \to & L^2(\mathbb{R}) \\
																u & \mapsto & \displaystyle \omega\mapsto \int_{\mathbb{R}} u(x) e^{i x \omega} \dx
\end{array}  \right. \]
where the right integral is defined by extending the operator defined on $L^1(\mathbb{R})\cap L^2(\mathbb{R})$. We also use the notation $\widehat{u} = \mathscr{F} u$.

\medskip

Now, we define the \it Shannon interpolation \rm, denoted  $\mathcal{I}$,  through the following diagram 
\begin{equation}
\label{def_Shannon_interpolation}
\xymatrixcolsep{5pc} \xymatrix{  L^2(\mathbb{Z})  \ar[r]^{ \mathcal{F} } \ar@/_1pc/[rrr]_{\mathcal{I}} & L^2(\mathbb{R} / 2\pi \mathbb{Z})  \ar[r]^{ u \mapsto \mathbb{1}_{(-\pi,\pi)} u }  & 
 L^2(\mathbb{R})  \ar[r]^{ \mathscr{F}^{-1} } &  L^2(\mathbb{R}) } ,
\end{equation}
where $\mathbb{1}_{(-\pi,\pi)}:\mathbb{R}\to \mathbb{R}$ the characteristic function of $(-\pi,\pi)$ and $ \mathscr{F}^{-1} $ is the inverse of the Fourier Plancherel transform.

\medskip

It is possible to deduce a very explicit formula to determine $\mathcal{I} \ug$ from $\ug$. Indeed, for $\ug\in L^2(\mathbb{Z})$ and $x\in \mathbb{R}$, we have 
$$
\mathcal{I} \ug (x)= \sum_{g \in \mathbb{Z}} \ug_g \textrm{sinc} (\pi (x-g) ), 
$$
where the sum converges in $L^2(\mathbb{R}) \cap L^{\infty}(\mathbb{R})$ and $\textrm{sinc}(x) = \frac{\sin(x)}x$, denotes the cardinal sine function.

In the following proposition, we give some properties of this interpolation useful to prove Theorem \ref{thm_growth} .
\begin{proposition} (see, for example Chapter $5.4$ in \cite{MR3156669}, for details)
\label{propo_Shannon}
\begin{itemize} 
\item $\mathcal{I}$ is an isometry, i.e.
\[ \forall \ug \in L^2(\mathbb{Z}), \ \sum_{g \in \mathbb{Z}} |\ug_g|^2 = \int_{\mathbb{R}} |\mathcal{I}\ug(x)|^2 \dx. \]  
\item The image of $\mathcal{I}$ is the set of functions whose Fourier support is a subset of $[-\pi,\pi]$. It is denoted by
\[ BL^2 := \mathcal{I}(L^2(\mathbb{Z})) =  \{ u \in L^2(\mathbb{R}) \ | \ {\rm Supp } \ \widehat{u} \subset [-\pi,\pi]\}. \]
\item If $\ug \in L^2(\mathbb{Z})$ then $\mathcal{I}\ug$ is an entire function  which  $\ug$ is the restriction on $\mathbb{Z}$, i.e.
\[  \forall g \in \mathbb{Z}, \ \left( \mathcal{I}\ug \right)(g) = \ug_g .\] 
\end{itemize}
\end{proposition}

\medskip

Now, we focus on properties more specific to the discrete Sobolev norms.
\begin{proposition} (see Proposition $2.6$ in \cite{ref_traveling_wave})
\label{prop_identify}
Let $\ug \in L^2(\mathbb{Z})$ be a sequence and let $u = \mathcal{I}\ug$ denote its Shannon interpolation. Then we have for almost all $\omega \in (-\pi,\pi)$
\[ \widehat{\mathcal{I}\Delta_1\ug}(\omega) = (2\cos(\omega) - 2) \widehat{u}(\omega)  = - 4 \left(\sin\left( \frac{\omega}2 \right)\right)^2 \widehat{u}(\omega) \]
and
\[ \widehat{\mathcal{I}|\ug|^2\ug}(\omega) = \sum_{k\in \mathbb{Z}} \widehat{|u|^2 u}(\omega + 2k\pi) = \sum_{k=-1}^1 \widehat{u}*\widehat{\bar{u}}*\widehat{u}(\omega + 2k\pi), \]
where $*$ is the usual convolution product.
\end{proposition}

\medskip

We deduce two important direct corollaries of this proposition. In the first one we identify the differential equation satisfied by the Shannon interpolation of a solution of DNLS.
\begin{corollary} Let $\ug \in C^1(\mathbb{R};L^2(\mathbb{Z}))$ be a solution of DNLS and let $u = \mathcal{I}\ug \in C^1(\mathbb{R};BL^2)$ denote its Shannon interpolation, then for all $t\in \mathbb{R}$ and almost all $\omega\in (-\pi, \pi)$,
\[ i\partial_t \widehat{u}(t,\omega) =  -4 \left(\sin\left( \frac{\omega}2 \right)\right)^2 \widehat{u}(\omega) +\nu \sum_{k=-1}^1 \widehat{u}*\widehat{\bar{u}}*\widehat{u}(\omega + 2k\pi).   \]
\end{corollary}
In the second corollary, we identify the discrete Sobolev norms.
\begin{corollary}
\label{corol_est_norms}
 Let $\ug \in L^2(\mathbb{Z})$ be a sequence, let $u = \mathcal{I}\ug$ denote its Shannon interpolation. If $n\in \mathbb{N}^*$ then
\[ \| \ug\|_{\dot H^n(\mathbb{Z})}^2 = \frac1{2\pi} \int 2^{2n} \left(\sin\left( \frac{\omega}2 \right)\right)^{2n} |\widehat{u}(\omega)|^2 \domega.   \]
Consequently, the continuous and discrete homogeneous Sobolev norms are equivalents, i.e.
\[ \left(  \frac2\pi \right)^n \|\partial_x^n u \|_{L^2(\mathbb{R})}  \leq  \| \ug\|_{\dot H^n(\mathbb{Z})} \leq \|\partial_x^n u \|_{L^2(\mathbb{R})}  .\]
\end{corollary}

\section{Proof of Theorem \ref{thm_growth}}

This section is devoted to the proof of Theorem \ref{thm_growth}. The idea is to construct some \it higher modified energies \rm controlling $H^n(h\mathbb{Z})$ norms and whose growth can be controlled by  $H^{n-1}(h\mathbb{Z})$ norms. The construction of \it higher modified energies \rm to study growth of Sobolev norms is a well known method (see \cite{MR1951312} or \cite{MR2996998}). 

\medskip

As explained at the end of the introduction, since inequality \eqref{ineq_growth} of Theorem \ref{thm_growth} is homogeneous, without loss of generality, we just need to prove it when $h=1$. 

\subsection{Construction of the modified energies}

DNLS is a Hamiltonian differential equation (see \cite{ref_traveling_wave}) whose Hamiltonian (i.e. its energy) is defined on $L^2(\mathbb{Z})$ by
\[ \HDNLS = \frac12 \| \cdot \|_{\dot H^1(\mathbb{Z})}^2 - \frac{\nu}4 \| \cdot \|_{L^4(\mathbb{Z})}^4.\]
So $\HDNLS(\ug)$ is a constant of the motion (it can be proven directly computing the discrete $L^2$ inner product of \eqref{def_DNLS} and $\ug(t)$). 

\medskip

If $u\in BL^2$ is the Shannon interpolation of a sequence $\ug\in L^2(\mathbb{Z})$ this Hamiltonian can be written as a function of $\widehat{u}$ (it is a consequence of Proposition \ref{prop_identify})
\begin{multline}
\label{hamil_four}
2\pi \HDNLS(\ug)= 
\frac12 \int \left(  2 \sin \frac{\omega}2 \right)^2 |\widehat{u}(\omega)|^2 \domega \\- \frac{\nu}4  \int_{w_1 +w_2 = w_{-1} + w_{-2} \mod 2\pi}  \widehat{u}(w_{1})\overline{\widehat{u}(w_{-1})}\widehat{u}(w_{2})\overline{\widehat{u}(w_{-2})}   {\rm d} w_1{\rm d} w_2{\rm d} w_{-1}.
\end{multline}
%

The principle of the construction of the modified energies is to change the weights of these integrals to get a control of high Sobolev norms. To explain this construction, we need to adopt more compact notations. Some of them are classical for NLS (see \cite{MR2996998}).

\medskip

 First, if $m\in \mathbb{N}^*$, we define $\mathcal{V}_{m}$ by
\[  \mathcal{V}_{m} := \left\{  w \in \mathbb{R}^{\llbracket  -m,m \rrbracket\setminus\{ 0\} } \ | \ \sum_{j=1}^m w_j-w_{-j} = 0 \mod 2\pi   \right\},\]
where $\llbracket -m,m \llbracket$ denotes the set $\{-m,\dots,m\}$,
and we equip it with its natural measure, denoted ${\rm d }w$, induced by the canonical Lesbegue measure of $\mathbb{R}^{2m}$.

\medskip

If $\mu \in L^{\infty}(\mathcal{V}_{m})$ and if $v\in L^2(\mathbb{R})$ is supported on $[-\pi,\pi]$, we define $\Lambda_{m}(\mu,v)$ by
\[ \Lambda_{m}(\mu,v) := \int_{\mathcal{V}_{m}} \mu(w)  \prod_{j=1}^m v(w_j)\overline{v(w_{-j})} {\rm d }w.   \]
To prove that $\Lambda_m$ is well defined, we just need to pay attention to the support of 
$$
w\mapsto \mu(w)  \prod_{j=1}^m v(w_j)\overline{v(w_{-j})}
$$ and to apply a convolution Young inequality (see Lemma \ref{lem_hehe} for details).

\medskip

For example, with this notation, we have a more compact expression of \eqref{hamil_four} given by
\begin{equation}
\label{formula_hamilt_comp}
 2\pi \HDNLS(\ug) = \frac12 \int \left(  2 \sin \frac{\omega}2 \right)^2 |\widehat{u}(\omega)|^2 \domega - \frac{\nu}4  \Lambda_2(\mathbb{1}_{\mathcal{V}_{2}}, \widehat{u}).
\end{equation}

\medskip

Then, we define a transformation $S_{m} : L^{\infty}(\mathcal{V}_{m}) \to L^{\infty}(\mathcal{V}_{m+1})$ by
\[ S_{m} \mu(w_{-m-1},w,w_{m+1}) = \sum_{k=1}^m  \mu(w + e_k (w_{m+1} - w_{-m-1}))  - \mu(w - e_{-k} (w_{m+1} - w_{-m-1})), \]
where $(e_k)_{k\in \llbracket  -m,m \rrbracket\setminus\{ 0\} }$ is the canonical basis of $ \mathbb{R}^{\llbracket  -m,m \rrbracket\setminus\{ 0\} }$.

\medskip

We define another transformation $D_m : L^{\infty}(\mathbb{R}) \to L^{\infty}(\mathcal{V}_{m})$ by
\[  D_{m} f(w) = \sum_{j=1}^m  f(w_j)  - f(w_{-j}).  \]

\medskip

We say that a function $\mu \in L^{\infty}(\mathcal{V}_{m})$ is $2\pi$ periodic with respect to each one of its variables, and we denote it by $\mu \in L^{\infty}_{\rm per}(\mathcal{V}_{m})$, if 
\[ \forall k \in \llbracket  -m,m \rrbracket\setminus\{ 0\},\  \mu(w+2\pi e_k) = \mu(w), \ w \ \rm{ a.e}.\]

\medskip

The following algebraic lemma explains why these notations are well suited to DNLS.
\begin{lemma} 
\label{lem_deriv_reso}
If $m \in \mathbb{N}^*$, $\mu \in L^{\infty}_{\rm per}(\mathcal{V}_{m})$ and $\ug \in C^1(\mathbb{R};L^2(\mathbb{Z}))$ is a solution of DNLS whose Shannon interpolation is denoted $u$, then we have
\begin{equation}
\label{id_algb}
i\partial_t \Lambda_{m}(\mu,\widehat{u}) = 2 \Lambda_{m}\left(\mu D_m \cos ,\widehat{u} \right) + \nu \Lambda_{m+1}(S_{m}\mu,\widehat{u}) .
\end{equation}
\end{lemma}
\begin{proof}
By definition, the quantity to identify can expanded as follow 
\begin{multline*}
 i\partial_t \Lambda_{m}(\mu,\widehat{u}) = \sum_{k=1}^m \int_{\mathcal{V}_{m}} \mu(w) \left[  \overline{\widehat{u}(w_{-k})}  i\partial_t \widehat{u}(w_k) + \widehat{u}(w_k)  i\partial_t \overline{\widehat{u}(w_{-k})}    \right] \prod_{j\neq k} \widehat{u}(w_j)\overline{\widehat{u}(w_{-j})} {\rm d }w \\ =: \sum_{k=1}^m I_k + I_{-k}.
\end{multline*}

\medskip

Now, we have to expand $I_k$ and $I_{-k}$ using the definition of DNLS. Applying Proposition \ref{prop_identify} we get
\[ \forall w_k \in (-\pi,\pi), \ i\partial_t \widehat{u}(w_k) = (2\cos w_k - 2) \widehat{u }(w_k) + \nu \sum_{\ell \in \mathbb{Z}} \widehat{|u|^2 u }(w_k + 2\pi \ell) .  \]
So, since $\mu$ is $2\pi$ periodic the direction $e_k$, we deduce
\begin{align*}
&I_k - \Lambda_{m} ((2\cos w_k - 2)\mu ,\widehat{u})  \\
&= \nu \int_{\mathcal{V}_{m}} \mu(w)   \overline{\widehat{u}(w_{-k})}  \left( \mathbb{1}_{w_k\in (-\pi,\pi)} \sum_{\ell \in \mathbb{Z}} \widehat{|u|^2 u }(w_k + 2\pi \ell) \right) \prod_{j\neq k} \widehat{u}(w_j)\overline{\widehat{u}(w_{-j})}  {\rm d }w \\
 		&=					 \nu  \int_{\mathcal{V}_{m}} \mu(w)   \overline{\widehat{u}(w_{-k})}  \widehat{|u|^2 u }(w_k)  \prod_{j\neq k} \widehat{u}(w_j)\overline{\widehat{u}(w_{-j})} {\rm d }w. 
\end{align*}
However, since for all $\omega \in \mathbb{R}$, $\overline{\widehat{u}}( \omega) = \widehat{\overline{u}}( -\omega)$, we have, for all $w_k\in \mathbb{R}$,
\[ \widehat{|u|^2 u }(w_k)  =  \int_{ w_{m+1} - w_{-m-1} + \widetilde{w}_k = w_k }    \widehat{u} (w_{m+1})\widehat{u}(\widetilde{w}_k) \overline{\widehat{u}}( w_{-m-1})  {\rm d} w_{m+1} {\rm d}w_{-m-1}  .\]
So, realizing the change of variable $w_k \leftarrow \widetilde{w}_k$, we get
\[  \int_{\mathcal{V}_{m}} \mu(w)   \overline{\widehat{u}(w_{-k})}  \widehat{|u|^2 u }(w_k)  \prod_{j\neq k} \widehat{u}(w_j)\overline{\widehat{u}(w_{-j})} {\rm d }w = \Lambda_{m+1}(\mu(w+ e_k (w_{m+1} - w_{m+1})),\widehat{u}) .\]

\medskip

Similarly, we could prove that
\[ I_{-k} = -\Lambda_{m} ((2\cos w_{-k} - 2)\mu ,\widehat{u}) - \nu \Lambda_{m+1}(\mu(w- e_{-k} (w_{m+1} - w_{m+1})),\widehat{u}) .\]
So, finally, we get
\begin{align*}
   i\partial_t \Lambda_{m}(\mu,\widehat{u}) &=  \sum_{k=1}^m I_k + I_{-k}  \\
   &= \Lambda_{m} \left( \sum_{k=1}^m  \left[ (2\cos w_{k} - 2)-(2\cos w_{-k} - 2) \right]\mu ,\widehat{u}\right) \\
    &+ \nu  \Lambda_{m+1}\left( \sum_{k=1}^m \mu(w+ e_k (w_{m+1} - w_{m+1})) - \mu(w- e_{-k} (w_{m+1} - w_{m+1})),\widehat{u} \right) \\
    &= 2 \Lambda_{m}\left(\mu D_m \cos ,\widehat{u} \right) + \nu \Lambda_{m+1}(S_{m}\mu,\widehat{u}).
\end{align*}
\end{proof}
\begin{corollary} 
\label{corol_algb_id}
Let $f \in L^{\infty}(\mathbb{R})$, let $\ug \in C^1(\mathbb{R};L^2(\mathbb{Z}))$ be a solution of DNLS and let $u$ be its Shannon interpolation. Then, we have
\[ \partial_t \int f(\omega) |\widehat{u}(\omega)|^2 \domega = \nu \frac{i}2 \Lambda_2 ( D_2 f  , \widehat{u}).\]
\end{corollary}
\begin{proof} The result only involves values of $f$ for $\omega\in (-\pi,\pi)$. So we can assume that $f$ is a $2\pi$ periodic function. Now, we observe that, by definition, we have
\[    \int f(\omega) |\widehat{u}(\omega)|^2 \domega = \Lambda_1 (f(w_1), \widehat{u}).\]
So, applying Lemma \ref{lem_deriv_reso}, we get
\[ \partial_t \int f(\omega) |\widehat{u}(\omega)|^2 \domega =-2 i\Lambda_1 \left( (D_1 \cos) f(w_1), \widehat{u} \right) - i \nu \Lambda_2( S_2\left[ f(w_1) \right] , \widehat{u}).  \]
Since $2\pi$ periodic functions clearly belong to the $D_1$ kernel, the first term is zero. So we just need to identify the second term. Indeed, paying attention to its symmetries and remembering that we have assumed that $f$ is $2\pi$ periodic function, we get
\begin{align*}
\Lambda_2( S_2\left[ f(w_1) \right] , \widehat{u}) &= \Lambda_2( f(w_1+ w_2 -w_{-2}) - f(w_1) , \widehat{u}) \\
										&= \Lambda_2( f(w_{-1}) - f(w_1) , \widehat{u}) \\
										&= -\frac12  \Lambda_2 ( f(w_1) + f(w_2) - f(w_{-1}) - f(w_{-2})  , \widehat{u}).
\end{align*}
\end{proof}

With these notations and results we can explain more precisely the construction of our \it higher modified energies \rm. But first, we explain why it is natural to introduce correction terms in the construction of our modified energy.

\medskip

In order to control the discrete $\dot{H}^n$ norm, it would seem natural to control its derivative. Indeed, if $\ug$ is a solution of DNLS and if $u$ is its Shannon interpolation, applying Corollary \ref{corol_algb_id} (and Corollary \ref{corol_est_norms}), we have
\begin{equation}
\label{jehaismondirecteurdethese} \partial_t  \| \ug \|_{\dot{H}^n(\mathbb{Z})}^2 = \nu \frac{i}{4\pi} \Lambda_2 \left( D_2 \left(  2 \sin \frac{\omega}2 \right)^{2n}  , \widehat{u} \right). 
\end{equation}
So a direct estimation of this derivative would naturally lead to (see Lemma \ref{lem_hehe} for a proof of this estimate) 
\[ \left| \partial_t  \| \ug \|_{\dot{H}^n(\mathbb{Z})}^2 \right|  \leq C \| \ug \|_{\dot{H}^n(\mathbb{Z})}^2 \|\ug\|_{\dot{H}^1(\mathbb{Z})}\|\ug\|_{L^2(\mathbb{Z})},\]
 where $C>0$ is an universal constant. 
Then assuming that the discrete homogeneous $H^1$ norm can be controlled uniformly on time by $M_{\ug(0)}$ (see Theorem \ref{ineq_growth} for the definition of $M_{\ug(0)}$ and the next subsection for a proof) and applying Gr\"onwall's inequality, we would get an universal constant $C>0$ such that, for all $t\geq 0$,
\[ \| \ug(t) \|_{\dot{H}^n(\mathbb{Z})} \leq  \| \ug(0) \|_{\dot{H}^n(\mathbb{Z})} e^{C M_{\ug(0)}^{\frac43} t}.\]
If we proceed by homogeneity to get a result depending on the stepsize $h$, we would get
\[ \| \ug(t) \|_{\dot{H}^n(h\mathbb{Z})} \leq  \| \ug(0) \|_{\dot{H}^n(h\mathbb{Z})} e^{C M_{\ug(0)}^{\frac43} t}.\]
Such a control is better than the trivial estimate \eqref{trivial_est} only for times shorter than $-\frac{n}C\log(h)$. So it is quite weak, if we compare it with the estimate of Theorem \ref{ineq_growth} because this later gives a non trivial control of $\| \ug(t) \|_{\dot{H}^n(\mathbb{Z})}$ for times shorter than $h^{- \frac{2n}{n-1}}$.

\medskip

So to improve this exponential bound, the idea of modified energy is to add a corrector term to $\| \ug \|_{\dot{H}^n(\mathbb{Z})}^2$ in order to cancel its time derivative \eqref{jehaismondirecteurdethese}. However, it turns out that there is an  algebraic obstruction to this construction as shows in  Lemma \ref{lem_constructor} below. For this reason, we consider another functional $\int f_n(\omega) |\widehat{u}(\omega)|^2 \domega$, where $f_n$ is a real function and such that this last quantity is equivalent to the square of the $\dot{H}^n(\mathbb{Z})$ norm. More precisely, observing the formula of the Hamiltonian (see \eqref{formula_hamilt_comp}), we consider a modified energy $E_n$ given by
\[ E_{n}(u) = \int f_n(\omega) |\widehat{u}(\omega)|^2 \domega +  \Lambda_2 (\mu_n,\widehat{u}),  \]
where $\mu_n \in L^{\infty}(\mathcal{V}_{2})$ is a function.

\medskip

Applying Lemma \ref{lem_deriv_reso} and its Corollary \ref{corol_algb_id}, if we want the correction term to cancel the derivative of $\int f_n(\omega) |\widehat{u}(\omega)|^2 \domega$ then $\mu_n$ has to solve the equation
\begin{equation}
\label{vulg_M}
 \nu D_2 f_n =  4\mu_n D_2 \cos  .
\end{equation}
Furthermore, if $\mu_n$ is a solution of \eqref{vulg_M}, we would have
\[ \partial_t E_{n}(u) = -i \nu \Lambda_3 (S_2 \mu_n, \widehat{u}). \]
With this construction, we will be able to prove Theorem \ref{thm_growth} by induction because we will prove that $ \Lambda_2 (\mu_n,\widehat{u})$ and $\Lambda_3 (S_2 \mu_n, \widehat{u})$ are controlled by the square of the $\dot{H}^{n-1}(\mathbb{Z})$ norm.

\medskip

Of course, we would like to iterate this process cancelling the derivative of $E_{n}(u)$ adding a new term to our modified energy. However, such a construction involve major algebraic issues and we do not know if it is possible (we should find some criteria of divisibility by $D_3 \cos$ on the ring of trigonometric polynomials on $\mathcal{V}_3$).

\medskip

To realize this strategy, we need, on the one hand, to design a function $\mu_n$ satisfying \eqref{vulg_M} without any singularity and, on the other hand, we need to control $\Lambda_2 (\mu_n,\widehat{u})$ and $\Lambda_3 (S_2 \mu_n, \widehat{u})$ by the square of the $\dot{H}^{n-1}(\mathbb{Z})$ norm. The two following lemmas treat each one of these issues.

\begin{lemma}
\label{lem_constructor}
If $f\in C^{\infty}(\mathbb{R})$ satisfies $f \unupoop{=}{\omega \to 0}{}  \mathcal{O}(\omega^{2n})$ , where $n\in \mathbb{N}^*$, and if $f$ is an even function and  $x\mapsto f(x-\frac{\pi}2)-f(\frac{\pi}2)$ is an odd function  then there exists $C>0$ such that we have
\[ \forall w \in \mathcal{V}_{2}, |D_2 f (w)| \leq C |D_2 \cos(w)| \sum_{j\in\{\pm 1,\pm2\}} w_{j}^{2n-2}.  \]
\end{lemma}
\begin{proof}
Since $f$ is an even function and  $x\mapsto f(x-\frac{\pi}2)-f(\frac{\pi}2)$ is an odd function, $f$ is a $2\pi$ periodic function whose Fourier series is
\[  f(\omega) = f(\frac{\pi}2) + \sum_{k\in \mathbb{N}} \beta_k \cos((2k+1)\omega) \textrm{ with } (b_k)_{k \in \mathbb{N}} \in \mathbb{R}^{\mathbb{N}}. \]
Furthermore, since $f$ is a $C^{\infty}$ function, for all $m\in \mathbb{N}^*$, there exists $C_m>0$ such that 
\[ \sum_{k\in \mathbb{N}} |\beta_k| (2k+1)^m\leq C_m.\]

\medskip

To get compact notations, we define the function $\cos_k$ (and similarly $\sin_k$) by
\[ \forall \omega\in \mathbb{R}, \ \cos_k \omega := \cos((2k+1)\omega) . \]

\medskip

If we assume that $w_1 + w_2 = w_{-1} +w_{-2} + 2\pi j$, with $j\in \mathbb{N}$ then we have
\begin{align*}
D_2 \cos_k w &= 2 \cos_k \left( \frac{w_1+w_2}2 \right) \cos_k \left( \frac{w_1-w_2}2  \right) - 2 \cos_k\left( \frac{w_{-1}+w_{-2}}2  \right)\cos_k\left( \frac{w_{-1}-w_{-2}}2 \right) \\
&=  2 \cos_k \left( \frac{w_1+w_2}2  \right) \left[ \cos_k\left( \frac{w_1-w_2}2  \right) - (-1)^j \cos_k\left( \frac{w_{-1}-w_{-2}}2 \right) \right] .
\end{align*}
But since $2k+1$ is odd, we have
\[ (-1)^j \cos_k\left( \frac{w_{-1}-w_{-2}}2 \right) = \cos_k \left(  \frac{w_{-1}-w_{-2}}2 + \pi j \right)  .\]
So, we get
\begin{align*}
& D_2 \cos_k w\\
 &= 2 \cos_k\left(  \frac{w_1+w_2}2  \right) \left[ \cos_k\left(  \frac{w_1-w_2}2  \right) -  \cos_k \left(  \frac{w_{-1}-w_{-2}}2 + \pi j \right) \right] \\
&= 4 \cos_k \left( \frac{w_1+w_2}2 \right) \sin_k \left(  \frac{ w_1-w_2 - w_{-1} + w_{-2} + 2\pi j }{4} \right) \sin_k \left(  \frac{ w_1-w_2 + w_{-1} - w_{-2} + 2\pi j }{4} \right).  
\end{align*}
However, we know that 
\[ \forall \omega \in \mathbb{R}, \ |\sin_k \omega| \leq (2k+1) |\sin(\omega)|. \]
Consequently, we can prove the same relation for $\cos_k$. Indeed, since $2k+1$ is an odd number, for all $\omega\in \mathbb{R}$, we have
\[  \left| \cos_k \left( \omega+\frac{\pi}2 \right) \right| = | \sin_k \omega | \leq (2k+1) |\sin(\omega)| = (2k+1)\left| \cos \left( \omega+\frac{\pi}2 \right) \right|.  \]
So we deduce that for all $w \in \mathcal{V}_{2}$, we have
\begin{equation*}
|D_2 \cos_k (w)| \leq (2k+1)^3 |D_2 \cos (w)|  .
\end{equation*}
Consequently, we have
\begin{equation}
\label{miracle}
   |D_2 f(w)| \leq C_3 |D_2 \cos(w)| .  
\end{equation}

\medskip

To conclude this proof, we just need to improve \eqref{miracle} when $w$ is small enough. In this case, we can forget the aliasing terms because if $\max_{j\in \{ \pm 1,\pm 2 \}} |w_j| <\frac{\pi}2$ then $w_1+w_2 =w_{-1}+w_{-2}$.

\medskip

Now we realize the following change of variable
\[  \left\{ \begin{array}{lll} X &=& \frac{w_1 - w_2 + w_3 -w_4}4, \\
				      Y & = & \frac{w_1 - w_2 - w_3 +w_4}4, \\
				      Z & = & \frac{w_1 + w_2}2, \\
				      H & = & w_1+w_2 - w_3 -w_4.
\end{array}\right.  \]
Then we define 
\[ F(X,Y,Z,H) = D_2 f(w).  \]

\medskip

Previously, we have proven that, for all $X,Y,Z\in \mathbb{R}$,
\[ F(X,Y,Z,0) = \sum_{k\in \mathbb{N}} \beta_k \cos_k Z \sin_k X \sin_k Y. \]
Consequently, we have
\[ F(0,Y,Z,0) = 0 \textrm{ and }  \partial_X F(X,0,Z,0)=0 . \]
So applying a Taylor expansion, we get
\begin{align*}
F(X,Y,Z,0) &= F(0,Y,Z,0) + X \int_{0}^1 \partial_X F(\alpha X,Y,Z,0)     {\rm d \alpha} \\
&=  X  \int_{0}^1 \int_{0}^1  \partial_X F(\alpha X,0,Z,0) + Y \partial_X \partial_Y F(\alpha X,\beta Y,Z,0)     {\rm d \beta}  {\rm d \alpha} \\
&=  XY\int_{0}^1 \int_{0}^1 \partial_X \partial_Y F(\alpha X,\beta Y,Z,0)     {\rm d \beta}  {\rm d \alpha} .
\end{align*}
 However, since $f \unupoop{=}{\omega \to 0}{}  \mathcal{O}(\omega^{2n})$, all the derivatives of $f$ of order less than $2n$ vanish in $0$. Thus, we deduce that all the derivative of $F$  of order less than $2n$ also vanish in $0$.  Consequently, realising a Taylor expansion, we get a constant $c>0$ such that if $|X|+|Y|+|Z|+|H|<1$ then
\[ | \partial_X \partial_Y F(X,Y,Z,H) | \leq c  ( |X|+|Y|+|Z|+|H| )^{2n-2}.  \]

So, if $|X|+|Y|+|Z|<1$, we have 
\[ |F(X,Y,Z,0)|  \leq c XY ( |X|+|Y|+|Z| )^{2n-2}.\]
Then we get
\[  |F(X,Y,Z,0)|  \leq \frac{c}{\cos 1 (\sinc 1)^2 } \cos Z \sin X \sin Y ( |X|+|Y|+|Z|)^{2n-2}.\]

\medskip

We can write this inequality with the variables $w_1,w_2,w_{-1},w_{-2}$. So,  since the norms are equivalent on $\{ w\in \mathbb{R}^{\{\pm 1,\pm 2 \}}, \ w_1+w_2 =w_{-1}+w_{-2}  \}$, there exists $\kappa>0$ such that for all $w\in \{ w\in \mathbb{R}^{\{\pm 1,\pm 2 \}}, \ w_1+w_2 =w_{-1}+w_{-2}  \}$, we have
$$
 \kappa^{-1} (|X|+|Y|+|Z|)^{2n-2} \leq \sum_{j\in \{\pm 1,\pm 2\}} |w_j|^{2n-2} \leq  \kappa (|X|+|Y|+|Z|)^{2n-2}.
$$  Thus, there exists $C \in (0,\frac{\pi}2)$ such that if $w\in \mathcal{V}_2$ satisfies $\max_{j\in \{\pm 1,\pm 2 \}} |w_j|<C^{-1}$ then
\begin{equation}
\label{eq_local}
|D_2 f (w)|\leq C  |D_2 \cos(w)| \sum_{j\in \{\pm 1,\pm 2\}} |w_j|^{2n-2}.
 \end{equation}

\medskip

Finally, to prove the lemma we just need to use \eqref{eq_local} when $w$ is small enough and \eqref{miracle} when it is large.
\end{proof}

\begin{lemma} Let $n,m\in \mathbb{N}$, $m\geq 2$. There exists $K>0$ such that for all $u\in BL^2$, we have 
\label{lem_hehe}
 \[ \Lambda_{m} \left( \sum_{j=1}^m w_j^{2n} +  w_{-j}^{2n}, |\widehat{u}| \right)  \leq K  \| \partial_x^n u \|_{L^2(\mathbb{R})}^2 \| \partial_x u \|_{L^2(\mathbb{R})}^{m-1}\| u \|_{L^2(\mathbb{R})}^{m-1}.  \]
\end{lemma}
\begin{proof}
This lemma is somehow a discrete integration by parts. By linearity, we just need to prove that
\[ \Lambda_{m} \left( |w_1|^{2n}, |\widehat{u}| \right)  \leq C \| \partial_x^n u \|_{L^2(\mathbb{R})}^2 \| \partial_x u \|_{L^2(\mathbb{R})}^{m-1}\| u \|_{L^2(\mathbb{R})}^{m-1}.  \]
Since, $\supp |\widehat{u}|  \subset \left[  -\pi,\pi\right]$ we have
\[ \Lambda_{m} \left( |w_1|^{2n}, |\widehat{u}| \right) = \sum_{k=1-m}^{m-1}  \int_{\mathcal{V}_{m,k}} w_1^{2n}  \prod_{j=1}^m |\widehat{u}(w_j)||\widehat{u}(w_{-j})|  {\rm d}w.   \]
where
\[  \mathcal{V}_{m,k} := \left\{  w \in \mathbb{R}^{\llbracket  -m,m \rrbracket\setminus\{ 0\} } \ | \ \sum_{j=1}^m w_j-w_{-j} = 2k\pi    \right\}.\]
So, applying Jensen's inequality to $x\mapsto x^n$, we get
\begin{align*}
 &\int_{\mathcal{V}_{m,k}} w_1^{2n}  \prod_{j=1}^m |\widehat{u}(w_j)||\widehat{u}(w_{-j})|  {\rm d}w\\
  &=  \int_{\mathcal{V}_{m,k}} |\omega_1|^{n} \left| w_{-1}+ \sum_{j=2}^m w_j - w_{-j} - 2k\pi \right|^n  \prod_{j=1}^m |\widehat{u}(w_j)||\widehat{u}(w_{-j})|   {\rm d}w \\
  &\leq (2m)^{n-1} \int_{\mathcal{V}_{m,k}} |w_1|^{n} \left(  |w_{-1}|^n + \sum_{j=2}^m |w_j|^n + |w_{-j}|^n + \left|  2k\pi \right|^n \right)   \prod_{j=1}^m |\widehat{u}(w_j)||\widehat{u}(w_{-j})|  {\rm d}w \\
  &= (2m)^{n-1}  \left[ (m-1)    (|\omega|^n |\widehat{u}|)^{*2}*|\widehat{u}|^{*{m-2}}*|\widehat{\bar{u}}|^{*{m}} + m (|\omega|^n |\widehat{u}|)*(|\omega|^n |\widehat{\bar{u}}|)  *|\widehat{u}|^{*{m-1}}*|\widehat{\bar{u}}|^{*{m-1}}   \right](2k\pi) \\
  &+ (2m)^{n-1} \int_{\mathcal{V}_{m,k}} |w_1|^{n} \left|  2k\pi \right|^n  \prod_{j=1}^m |\widehat{u}(w_j)||\widehat{u}(w_{-j})| {\rm d}w .
\end{align*}
The first term can be estimated by an elementary Young convolution inequality to get
\begin{align*}
&\left[ (m-1)    (|\omega|^n |\widehat{u}|)^{*2}*|\widehat{u}|^{*{m-2}}*|\widehat{\bar{u}}|^{*{m}} + m (|\omega|^n |\widehat{u}|)*(|\omega|^n |\widehat{\bar{u}}|)  *|\widehat{u}|^{*{m-1}}*|\widehat{\bar{u}}|^{*{m-1}}   \right](\frac{2k\pi}h) \\
&\leq (2m-1) \| \omega^n \widehat{u} \|_{L^2(\mathbb{R})}^2 \| \widehat{u} \|_{L^1(\mathbb{R})}^{2m-2}.  
\end{align*}
The second term is an aliasing term. If $k=0$, this term is $0$, so we can assume $k\neq 0$. Now observe that if the sum of $2m$ numbers, all smaller than $1$ is larger than $2$ then at least $2$ of them are larger to $\frac{1}{2m-1}$. 
Consequently, applying the same Young convolution inequality, we have
\begin{align*}
&\int_{\mathcal{V}_{m,k}} |w_1|^{n} \left|  2k\pi \right|^n  \prod_{j=1}^m |\widehat{u}(w_j)||\widehat{u}(w_{-j})|  {\rm d}w  \\
&\leq \int\limits_{\substack{w \in \mathcal{V}_{m,k} \\ |w_{-1}|\geq \frac{\pi}{2m-1}}  }  |w_1|^{n} \left|  2k\pi \right|^n  \prod_{j=1}^m |\widehat{u}(w_j)||\widehat{u}(w_{-j})|  {\rm d}w  + \int\limits_{\substack{ w \in \mathcal{V}_{m,k} \\ |w_{2}|\geq \frac{\pi}{2m-1}}   } |\omega_1|^{n} \left|  2k\pi \right|^n  \prod_{j=1}^m |\widehat{u}(w_j)||\widehat{u}(w_{-j})| {\rm d}w  \\
&\leq (2 |k| (2m-1))^n \int_{\mathcal{V}_{m,k}}  |w_1|^{n}   ( |\omega_2|^n + |\omega_{-1}|^n  )   \prod_{j=1}^m |\widehat{u}(w_j)||\widehat{u}(w_{-j})| {\rm d}w  \\
&\leq 2 (2 |k| (2m-1))^n \| \omega^n \widehat{u} \|_{L^2(\mathbb{R})}^2 \| \widehat{u} \|_{L^1(\mathbb{R})}^{2m-2}.
\end{align*} 
To conclude rigorously this proof, we just need to control classically $\| \widehat{u} \|_{L^1(\mathbb{R})}^2$ by $\| u\|_{L^2(\mathbb{R})}\| \partial_x u\|_{L^2(\mathbb{R})}$. Indeed, if $v\in H^1(\mathbb{R})$, using Cauchy Schwarz inequality, we get
\[ \| \widehat{v} \|_{L^1(\mathbb{R})} \leq \sqrt{2\pi} \| \sqrt{1+\omega^2} \widehat{v} \|_{L^2(\mathbb{R})} = 2\pi \|  v\|_{H^1(\mathbb{R})}.\]
So, optimizing this inequality with respect to $\lambda$ through the transformation $v \rightarrow v(\lambda x) $, we get
\[ \| \widehat{v} \|_{L^1(\mathbb{R})} \leq  \sqrt 8 \pi \| v\|_{L^2(\mathbb{R})}\| \partial_x v\|_{L^2(\mathbb{R})}.\]
\end{proof}

\medskip

\subsection{Proof of Theorem \ref{thm_growth} by induction}
With all these tools, now, we prove Theorem \ref{thm_growth}. As explained at the beginning of this section, we just need to focus on the case $h=1$.
 We are going to proceed by induction. 

\medskip

\begin{itemize}
\item We focus on the case $n=1$. Let $\ug \in C^1(\mathbb{R};L^2(\mathbb{Z}))$ be a solution of DNLS. Since $\HDNLS$ is a constant of the motion of DNLS, for all $t\in \mathbb{R}$, we have
\begin{equation}
\label{ener_conserv}
 \| \ug(t) \|_{\dot{H}^1(\mathbb{Z})}^2 - \frac{\nu}2 \| \ug(t)\|_{L^4(\mathbb{Z})}^4 = \| \ug(0) \|_{\dot{H}^1(\mathbb{Z})}^2 - \frac{\nu}2 \| \ug(0)\|_{L^4(\mathbb{Z})}^4 .   
\end{equation}
Since $\|\ug\|_{L^2(\mathbb{Z})}^2$ is also a constant of the motion, we have
\[  \| \ug(t)\|_{L^4(\mathbb{Z})}^4 \leq \| \ug(0)\|_{L^2(\mathbb{Z})}^2 \| \ug(t) \|_{L^{\infty}(\mathbb{Z})}^2. \]
Let $u$ be the Shannon interpolation of $\ug$. Since $u_{| \mathbb{Z}}=\ug$ (see Proposition \ref{propo_Shannon}), we have
\[ \| \ug(t) \|_{L^{\infty}(\mathbb{Z})}^2 \leq  \| u(t) \|_{L^{\infty}(R)}^2 \leq c \| \partial_x u(t) \|_{L^2(\mathbb{R})} \| u(t)\|_{L^2(\mathbb{R})}, \]
where $c$ is an universal constant associated to the classical Sobolev embedding. Since Shannon interpolation is an isometry we have proven that
\[ \| \ug(t)\|_{L^4(h\mathbb{Z})}^4 \leq c \| \ug(0)\|_{L^2(\mathbb{Z})}^3  \| \partial_x u(t) \|_{L^2(\mathbb{R})}  .\]
Now applying the estimate of Corollary \eqref{corol_est_norms}, we get a \it discrete Gagliardo-Nirenberg \rm inequality (for a sharper version of this inequality see Lemma $3.4$ in \cite{MR3222131})
\[   \| \ug(t)\|_{L^4(\mathbb{Z})}^4 \leq \frac{2c}{\pi} \| \ug(0)\|_{L^2(\mathbb{Z})}^3 \| \ug(t) \|_{\dot{H}^1(\mathbb{Z})} .\]

\medskip

Applying this inequality to \eqref{ener_conserv}, we get
\[   \| \ug(t) \|_{\dot{H}^1(\mathbb{Z})}^2 - \frac{c}{\pi} \| \ug(0)\|_{L^2(\mathbb{Z})}^3 \| \ug(t) \|_{\dot{H}^1(\mathbb{Z})} \leq \| \ug(0) \|_{\dot{H}^1(\mathbb{Z})}^2 .\]
Consequently, we have proven that
\begin{align*}
 \| \ug(t) \|_{\dot{H}^1(\mathbb{Z})} &\leq  \frac{c}{2\pi} \| \ug(0)\|_{L^2(\mathbb{Z})}^3 + \frac12\sqrt{ \left(  \frac{c}{\pi} \| \ug(0)\|_{L^2(\mathbb{Z})}^3\right)^2 + 4  \| \ug(0) \|_{\dot{H}^1(\mathbb{Z})}^2  } \\
							&\leq  C \left( \| \ug(0) \|_{\dot{H}^1(\mathbb{Z})} +   \| \ug(0)\|_{L^2(\mathbb{Z})}^3  \right), 
\end{align*}
with $C = \max(1,\frac{c}\pi)$.

\item Let $n\geq 2$, let $\ug \in C^1(\mathbb{R};L^2(\mathbb{Z}))$ be a solution of DNLS satisfying for all $t\in \mathbb{R}$
\begin{equation}
\label{induc_hyp}
\| \partial_x^{n-1} u(t) \|_{L^2(\mathbb{R})}^2 \leq  C  \left( \| \partial_x^{n-1} u_0 \|_{L^2(\mathbb{R})}^2 + M_{u_0}^{\frac{4n-2}3}  +  |t|^{n-2} M_{u_0}^{\frac{8n-10}3}\right) ,  
\end{equation}
where $u$ is the Shannon interpolation of $u$ and
\[ M_{u(0)} =  \| \partial_x u_0 \|_{L^2(\mathbb{R})} + \| u_0  \|_{L^2(\mathbb{R})}^3    .\]
Here, it is easier to work with an inequality on $u$ instead of $\ug$ but applying  the estimate of Corollary \eqref{corol_est_norms}, \eqref{induc_hyp} is equivalent to the inequality of Theorem \ref{thm_growth}.

\medskip

First, we are going to construct our modified energy with Lemma \ref{lem_constructor}. So we have to choose our function $f_n$. This function has to satisfy some criteria. First, we want  $\int f_n |\widehat{u}(\omega)|^2 \domega$ to be equivalent to square of the homogeneous $H^n$ norm of $u$. So we are looking for a regular function $f_n$ such that 
\begin{equation}
\label{equiv_norm_micro}
\forall \omega \in (-\pi,\pi), \ \alpha  \omega^{2n} \leq f_n(\omega) \leq \alpha^{-1} \omega^{2n}.  
\end{equation}
Second, we want $f_n-f_n(\frac{\pi}2)$ to be even in $0$ and odd in $\frac{\pi}2$. So we cannot choose $f_n(\omega) = \omega^{2n}$ or $f_n(\omega) = \left(2\sin(\frac{\omega}2)\right)^{2n}$. To satisfy these symmetries it is natural to look for $f_n$ as a trigonometric polynomial.

\medskip

By performing an analysis involving elementary linear algebra, we find that  $f_n$ defined by 
\begin{equation}
\label{choice_energy}
 f_n(\omega) := 1 - \cos(\omega)  \sum_{k=0}^{n-1} \frac{C_{2k}^k}{4^{k}} (\sin \omega)^{2k}, 
\end{equation}
is the trigonometric polynomial of minimal degree (and such  $f(\frac{\pi}{2}) = 1$) satisfying the previous hypothesis. 
Indeed, by construction, $f_n-1$ is even in $0$ and odd in $\frac{\pi}2$. Furthermore, in $\mathbb{R}\llbracket X \rrbracket$ (i.e. formally), we have (see, for example, formula $3.6.9$ in \cite{MR0167642})
\[ \frac1{\sqrt{1 - X^2}} =   \sum_{k\in \mathbb{N}} \frac{C_{2k}^k}{4^{k}} X^{2k}.\]
Since, for all $\omega\in (-\frac{\pi}2,\frac{\pi}2)$, $\cos \omega= \sqrt{1 - (\sin \omega)^2}$, we deduce that
\[   f_n(\omega) = \cos(\omega)  \sum_{k\geq n} \frac{C_{2k}^k}{4^{k}} (\sin \omega)^{2k} .\]
Consequently, we get $f_n>0$ on $\omega\in (0,\frac{\pi}2)$ and $f_n(\omega) \unupoop{\sim}{\omega \to 0}{}  \frac{C_{2n}^n}{4^{n}} \omega^{2n}$. So, using the symmetries of $f_n$, we deduce that there exists $\alpha >0$ such that \eqref{equiv_norm_micro}  is satisfied.

\medskip

Then we define on $\mathcal{V}_2$ a function $\mu_n \in L^{\infty}(\mathcal{V}_2)$ by
\[   \mu_n = \frac{\nu}4 \frac{D_2 f_n}{ D_2 \cos}. \]
In Lemma \ref{lem_constructor}, we have proven  that $\mu_n$ is well defined as a $L^{\infty}(\mathcal{V}_2)$ function (in fact, we could have proven that it is a regular function). Furthermore, we have proven that for all $w\in \mathcal{V}_2$, we have
\begin{equation}
\label{control_Mnh}
\left| \mu_n(w) \right| \leq C_n \sum_{j\in \{ \pm 1,\pm 2 \}} |w_j|^{2n-2},
\end{equation}
where $C_n$ depends only of $n$.

\medskip

Then we define our modified energy, for $v\in BL^2_1$ by
\[ E_{n}(v) :=  \int_{\mathbb{R}} f_{n}(\omega ) |\widehat{v}(\omega)|^2 \domega + \Lambda_2(\mu_{n},\widehat{v}). \]
Applying \eqref{equiv_norm_micro}, we get, for all $t\in \mathbb{R}$,
\begin{align*}
2\pi \alpha \| \partial_x^n u(t) \|_{L^2(\mathbb{R})}^2 &=\alpha  \int \omega^{2n} |\widehat{u}(t,\omega)|^2 \domega \\
							   &\leq   \int f_n(\omega ) |\widehat{u}(t,\omega)|^2 \domega \\
							   &\leq |E_{n}(u_0)| + |\Lambda_2(\mu_n,\widehat{u}(t))| + \int_{0}^t |\partial_s E_{n}(u(s))| \ds.
\end{align*}
To conclude the induction step we have to control each one of these terms.

\begin{itemize}
\item First, we focus on $\displaystyle \int_{0}^t |\partial_s E_{n}(u(s))| \ds.$

\medskip

Applying Lemma \ref{lem_deriv_reso}, we get
\[ \partial_t E_{n}(u(t)) = -i\nu \Lambda_{3}(S_2 \mu_{n},\widehat{u(t)}).  \]
So, applying \eqref{control_Mnh} and Jensen's inequality to $x\mapsto x^{2n-2}$, we get
\[   |\partial_t E_{n}(u(t))| \leq C_n  3^{2n-3} 4 \Lambda_{3}( \sum_{j=1}^3 w_{j}^{2n-2} + w_{-j}^{2n-2} , |\widehat{u(t)}|  ). \]
Consequently, applying Lemma \ref{lem_hehe}, we get a constant $K_n>0$ such that
\[ |\partial_t E_{n}(u(t))| \leq  K_n \| \partial_x^{n-1} u(t) \|_{L^2(\mathbb{R})}^2 \| \partial_x u(t) \|_{L^2(\mathbb{R})}^{2} \| u(t) \|_{L^2(\mathbb{R})}^{2}.  \]
However, as we have proven at the initial step, there exists an universal constant $c>0$ such that
\[ \forall t\in \mathbb{R}, \ \| \partial_x u(t) \|_{L^2(\mathbb{R})}^{2} \| u(t) \|_{L^2(\mathbb{R})}^{2} \leq c M_{u_0}^{\frac{8}3}  .\]
So, from the induction hypothesis (see \eqref{induc_hyp}), we get
\[  |\partial_t E_{n}(u(t))| \leq \kappa  \left( \| \partial_x^{n-1} u_0 \|_{L^2(\mathbb{R})}^2 M_{u_0}^{\frac{8}3} + M_{u_0}^{\frac{4n+6}3}  +  |t|^{n-2} M_{u_0}^{\frac{8n-2}3}\right),\]
with $\kappa = cCK_n$. Consequently, we have
\[ \left| \int_{0}^t |\partial_s E_{n}(u(s)) | \ds  \right| \leq   \kappa  \left( |t| \| \partial_x^{n-1} u_0 \|_{L^2(\mathbb{R})}^2 M_{u_0}^{\frac{8}3} + |t| M_{u_0}^{\frac{4n+6}3}  +  \frac{|t|^{n-1}}{n-1} M_{u_0}^{\frac{8n-2}3}\right).  \]
It is almost the required estimate of the induction. In fact, we just need to modify it using Young inequalities. Indeed, on the one hand we have
\[ |t| M_{u_0}^{\frac{4n+6}3} \leq   \frac{|t|^{n-1}}{n-1} M_{u_0}^{\frac{8n-2}3} + \frac{n-2}{n-1}  M_{u_0}^{\frac{4n+2}3}.\]
On the other hand, since, by H\"older inequality,  
\begin{equation}
\label{res_Holder}
 \| \partial_x^{n-1} u_0 \|_{L^2(\mathbb{R})}^2 \leq \| \partial_x^{n} u_0 \|_{L^2(\mathbb{R})}^{2\frac{n-2}{n-1}}  \| \partial_x u_0 \|_{L^2(\mathbb{R})}^{\frac{2}{n-1}} , 
 \end{equation}
we have
\[ |t| \| \partial_x^{n-1} u_0 \|_{L^2(\mathbb{R})}^2 M_{u_0}^{\frac{8}3}  \leq  |t| \| \partial_x^{n} u_0 \|_{L^2(\mathbb{R})}^{2 \frac{n-2}{n-1}}  M_{u_0}^{\frac{8}3+\frac{2}{n-1} } \leq \frac{n-2}{n-1} \| \partial_x^{n} u_0 \|_{L^2(\mathbb{R})}^2 + \frac{|t|^{n-1}}{n-1}  M_{u_0}^{\frac{8n-2}3} .\]

\item Second, we focus on $|\Lambda_2(\mu_n,\widehat{u}(t))|$.

\medskip

Here, we just need to apply \eqref{control_Mnh} to get
\[ |\Lambda_2(\mu_n,\widehat{u}(t))| \leq C_n  \Lambda_2(\sum_{j=1}^2 |w_j|^{2n-2} +|w_{-j}|^{2n-2}   ,|\widehat{u}(t)|). \]
So, we deduce of Lemma \ref{lem_hehe}, that there exists $\kappa_n>0$ such that
\[ |\Lambda_2(\mu_n,\widehat{u}(t))| \leq \kappa_n  \| \partial_x^{n-1} u(t) \|_{L^2(\mathbb{R})}^2 \| \partial_x u(t) \|_{L^2(\mathbb{R})} \| u(t) \|_{L^2(\mathbb{R})}. \]
Consequently, applying the induction hypothesis (see \eqref{induc_hyp}), and the initial step, we have 
\[  |\Lambda_2(\mu_n,\widehat{u}(t))| \leq  K   \left( \| \partial_x^{n-1} u_0 \|_{L^2(\mathbb{R})}^2 M_{u_0}^{\frac{4}3} + M_{u_0}^{\frac{4n+2}3}  +  |t|^{n-2} M_{u_0}^{\frac{8n-6}3}\right) , \]
with $K=C\kappa_n c$ where $c$ is an universal constant. \\
As previously, we need to apply some Young inequalities to modify this estimate to get the induction estimate. On the one hand, we have
\[ |t|^{n-2} M_{u_0}^{\frac{8n-6}3} \leq   \frac{n-2}{n-1} t^{n-1}  M_{u_0}^{\frac{8n-2}3}  + \frac1{n-1} M_{u_0}^{\frac{4n+2}3} .  \]
On the other hand, applying \eqref{res_Holder}, we get
\[  \| \partial_x^{n-1} u_0 \|_{L^2(\mathbb{R})}^2 M_{u_0}^{\frac{4}3} \leq  \| \partial_x^{n} u_0 \|_{L^2(\mathbb{R})}^{2 \frac{n-2}{n-1}}  M_{u(0)}^{\frac{4}3 + \frac{2}{n-1}} \leq  \frac{n-2}{n-1}\| \partial_x^{n} u_0 \|_{L^2(\mathbb{R})}^{2} + \frac1{n-1} M_{u_0}^{\frac{4n+2}3} .\]

\item Finally, we focus on $|E_{n}(u_0)|$.

\medskip

We apply the triangle inequality to get
\[   |E_{n}(u_0)| \leq   \int f_n(\omega ) |\widehat{u_0}(\omega)|^2 \domega + |\Lambda_2(\mu_n,\widehat{u_0})|.\]
On the one hand, applying \eqref{equiv_norm_micro}, we get
\[   \int f_n(\omega ) |\widehat{u_0}(\omega)|^2 \domega  \leq \alpha^{-1} \int  \omega^{2n} |\widehat{u}(\omega)|^2 \domega = 2\pi \alpha^{-1}\| \partial_x^{2n} u_0 \|_{L^2(\mathbb{R})}^2 .\]
On the other hand, applying the estimate of $\Lambda_2(\mu_n,\widehat{u}(t))$, when $t=0$, we get
\[  |\Lambda_2(\mu_n,\widehat{u_0})| \leq   K   \left(\frac{n-2}{n-1}\| \partial_x^{n} u_0 \|_{L^2(\mathbb{R})}^{2} + \left[1+\frac1{n-1}\right] M_{u_0}^{\frac{4n+2}3}  \right).\]

\end{itemize}

\end{itemize}

\bibliographystyle{plain}
\bibliography{sobolev}

\end{document}